\documentclass{article}
\usepackage[utf8]{inputenc}
\usepackage{amsmath}
\usepackage{amsfonts}
\usepackage{amsthm}
\usepackage{amssymb}
\usepackage{mathtools}
\usepackage{tikz-cd}
\usepackage{float}
\usepackage{verbatim}

\usepackage{hyperref}
\usepackage[capitalize]{cleveref}

    \renewcommand{\ll}{\left\langle}
    \newcommand{\rr}{\right\rangle}
    \newcommand{\ls}{\left\{}
    \newcommand{\rs}{\right\}}

    \newcommand{\id}{\operatorname{id}}
    \newcommand{\mcH}{\ensuremath{\mathcal{H}}}
    
    \newcommand{\mbR}{\ensuremath{\mathbb{R}}}
    
    \newcommand{\mcP}{\ensuremath{\mathcal{P}}}
    \newcommand{\mcQ}{\ensuremath{\mathcal{Q}}}
      
    \newcommand{\Sym}[1][n]{\operatorname{Sym}_{ #1 }}
    \newcommand{\CoxComplex}[1][\Sym]{\ensuremath{\Sigma( #1 )}}
    \newcommand{\targetspace}{X}
    \newcommand{\subtargetspace}{Y}

\newcommand\restr[2]{{
  \left.\kern-\nulldelimiterspace 
  #1 
  \vphantom{\big|} 
  \right|_{#2} 
  }}

\numberwithin{thmcounter}{section}
\newtheorem{theorem}{Theorem}[section]
\newtheorem{corollary}[theorem]{Corollary}
\newtheorem{lemma}[theorem]{Lemma}
\newtheorem{proposition}[theorem]{Proposition}

\theoremstyle{definition}
\newtheorem{definition}[theorem]{Definition}
\newtheorem{remark}[theorem]{Remark}
\newtheorem{example}[theorem]{Example}

\title{Stratifying the space of barcodes using Coxeter complexes}
\author{Benjamin Brück, Adélie Garin}

\begin{document}

\maketitle

\begin{abstract}
    We use tools from geometric group theory to produce a stratification of the space $\mathcal{B}_n$ of barcodes with $n$ bars. The top-dimensional strata are indexed by permutations associated to barcodes as defined by Kanari, Garin and Hess. More generally, the strata correspond to marked double cosets of parabolic subgroups of the symmetric group $\Sym$. This subdivides $\mathcal{B}_n$ into regions that consist of barcodes with the same averages and standard deviations of birth and death times and the same permutation type. We obtain coordinates that form a new invariant of barcodes, extending the one of Kanari--Garin--Hess. This description also gives rise to metrics on $\mathcal{B}_n$ that coincide with modified versions of the bottleneck and Wasserstein metrics.
\end{abstract}

\section{Introduction}
Barcodes \cite{topdata,perssurvey,barcodes} are topological summaries of the persistent homology of a filtered space. 
The barcode $B$ associated to a \emph{filtration} $\{X_t\}_{t \in \mbR}$ is a multiset of points $(b,d) \in \mbR^2$. It summarises the creation and destruction of homology classes while varying the parameter $t$, which is often interpreted  as ``time''.
A bar $(b,d)\in B$  corresponds to a homology cycle appearing in $X_{b}$ and becoming a boundary in $X_{d}$.
The first element of the pair $(b,d)$ is called the \emph{birth} and the second one the \emph{death}.

Persistent homology has applications in many fields, from biology \cite{ Byrne2019,Gameiro2015,TMD,Reimann2017} to material science \cite{Robins_materials, Lee2018, robins_percolating_2016}, astronomy \cite{astronomy} and climate science~\cite{Muszynski2019}. In many of these applications, it is necessary to study statistics on barcodes. Unfortunately, the space of barcodes is not a Hilbert space, which means that it can be difficult to apply statistical methods to it. 
Several ways to overcome the issue exist, such as the creation of kernels to map barcodes into a Hilbert space \cite{landscapes, stable_signatures,pers_images, diagram_to_vectors}.

In this paper, we tackle this issue from a different perspective. We use combinatorial tools from geometric group theory to define new coordinates for describing barcodes. These coordinates divide the space of barcodes into regions indexed by the averages and the standard deviations of births and deaths and by the permutation type of a barcode as defined in \cite{TRN,trees_barcodesII}. By associating to a barcode the coordinates of its region, we define a new invariant of barcodes.
This opens the door to doing statistics on barcodes using methods from the field of permutation statistics.

\paragraph{Motivation}

The motivation for this work is to understand the space of barcodes from a combinatorial and geometric point of view. We call a barcode \emph{strict} if there are no two pairs in it that have the same birth or death. It was observed in \cite{TRN} that to a strict barcode $B = \ls (b_i,d_i) \rs_{i \in \ls 1, \ldots, n\rs}$ with $n$ bars, one can associate a permutation $\sigma_B\in \Sym$. It is the permutation such that the bar with the $i$-th smallest death has the $\sigma_B(i)$-th smallest birth. 
This divides the set of strict barcodes with $n$ bars into $n!$ equivalence classes, one for each element of the symmetric group $\Sym$. Based on this observation, one can study the combinatorial properties of strict barcodes by describing these equivalence classes---or equivalently, the elements of $\Sym$---and the relations between them.

A first approach to this, taken in \cite{TRN,trees_barcodesII}, is to consider the Cayley graph of the symmetric group with respect to the generating set given by adjacent transpositions $(i,i+1)$. 
This yields a combinatorial representation of the elements of $\Sym$. It tells us how a pair of permutations can be transformed into one another using transpositions one step at a time.  However, it yields no information about ``higher order relations'' that exist among larger sets of permutations.

\begin{figure}
    \centering
    \includegraphics[scale=0.45]{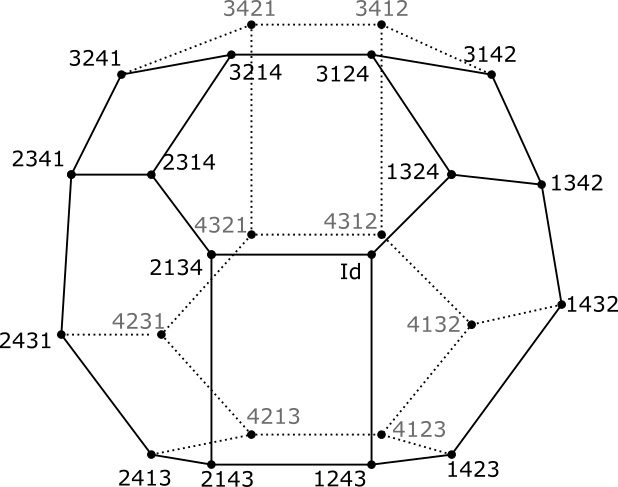}
    \caption{The permutohedron \cite{permutohedron} of order $4$ is a polyhedral decomposition of the sphere where each vertex corresponds to an element of the symmetric group $\Sym[4]$. Its $1$-skeleton is the Cayley graph of $\Sym[4]$ (see also \cref{fig_cayley_s4}).}
    \label{fig_permutohedron}
\end{figure}
A way to resolve this is to add higher dimensional cells to the Cayley graph and to consider it more geometrically as a cell complex instead of as a (combinatorial) graph. 
A first approach would be to use that the Cayley graph of $\Sym$ is the $1$-skeleton of the permutohedron \cite{permutohedron} of order $n$, see \cref{fig_permutohedron}.
This observation embeds the Cayley graph into a polyhedral decomposition of the $(n-2)$-sphere. As this is a more geometric object, it allows to continuously ``walk'' from one permutation to another. 
The problem is that only the vertices (and not the higher dimensional cells) of the permutohedron have an interpretation in terms of elements of the symmetric group.
Furthermore, this representation lacks a notion of ``size'' for barcodes. For instance, the two barcodes depicted in \cref{barcode_size} lie in the same equivalence class, i.e.~have the same associated permutation.
\begin{figure}
    \centering
    \includegraphics[scale=0.5]{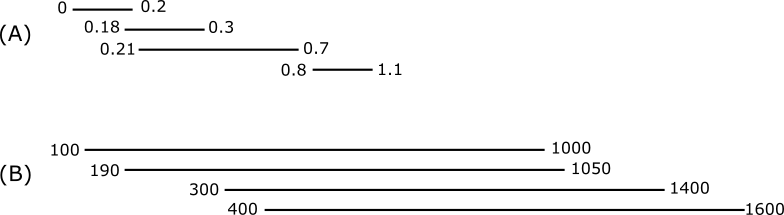}
    \caption{Two barcodes with the same associated permutation (the identity $[1234]$) but with large differences in their birth and death values.}
    \label{barcode_size}
\end{figure}

The alternative that we suggest to overcome these problems is to work with \emph{Coxeter complexes} instead of  permutohedra. The Coxeter complex associated to $\Sym$ is the dual of the permutohedron of order $n$ (see \cref{fig_permutohedron_dual}). It forms a simplicial decomposition of the $(n-2)$-sphere and is well-studied in the context of reflection groups and Tits buildings.
\begin{figure}
    \centering
    \includegraphics[scale=0.4]{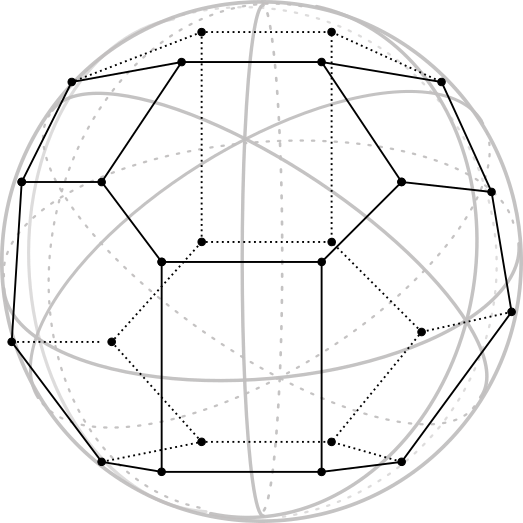}
    \caption{The permutohedron of order $4$ (black) is the dual of the Coxeter complex $\CoxComplex[{\Sym[4]}]$ (grey).}
    \label{fig_permutohedron_dual}
\end{figure}
For us, it has the advantage that its top-dimensional simplices correspond in a natural way to permutations and only passing through a face of lower dimension changes such a permutation. This allows for a better description of continuous changes between different permutations. It also has the advantage that it comes with an embedding in $\mbR^n$, where the additional two real parameters that are needed to describe positions relative to this $(n-2)$-dimensional space have a natural interpretation in terms of the ``size'' of barcodes.
Moreover, using the Coxeter complex description for barcodes allows to define the permutation type of \emph{any} barcode. For non-strict barcodes, it is defined only up to \emph{parabolic subgroups} of $\Sym$, i.e.~subgroups that are generated by sets of adjacent transpositions.

\paragraph{Contributions}

In this paper, we use Coxeter complexes to develop a description of the set $\mathcal{B}_n$ of barcodes with $n$ bars with coordinates that have natural interpretations when doing statistics with barcodes. These coordinates define a stratification of $\mathcal{B}_n$ where the top-dimensional strata are indexed by the symmetric group $\Sym$. Our main contributions can be summarised as follows.

\begin{theorem}
Let $\mathcal{B}_n$ denote the set of barcodes with $n$ bars.
\begin{enumerate}
    \item $\mathcal{B}_n$ can in a natural way be seen as a subset of a quotient $\Sym \backslash \mbR^{2n}$.
    \item $\mathcal{B}_n$ is stratified over the poset of marked double cosets of parabolic subgroups of $\Sym[n]$.
    \item Using this description, one obtains a decomposition of $\mathcal{B}_n$ into different regions. Each region is characterised as the set of all barcodes having the same average birth and death, the same standard deviation of births and deaths and the same permutation type $\sigma_B\in \Sym$.
    \item This description gives rise to metrics on $\mathcal{B}_n$ that coincide with modified versions of the bottleneck and Wasserstein metrics.
\end{enumerate}
\end{theorem}
For more detailed and formal statements of these results, see \cref{prop_bijection_B_n}, \cref{cor_stratification}, \cref{cor_simplified_coordinates} and \cref{prop_isometry_B_n}.

To obtain this description of $\mathcal{B}_n$ we proceed as follows. A barcode is an (unordered) multiset of $n$ pairs of real numbers (births and deaths). It can hence be seen as a point in the quotient space $\Sym \backslash (\mbR^n \times \mbR^n$), where the action of $\Sym$ permutes the coordinate pairs. Since the birth is smaller than the death for every barcode, $\mathcal{B}_n$ is a proper subset of this quotient of $\mbR^{2n}$.

The Coxeter complex $\CoxComplex$ associated to $\Sym$ is a simplicial complex whose geometric realisation is homeomorphic to an $(n-2)$-sphere. Hence, we can decompose $\mbR^n$ as 
\begin{equation*}
    \mbR^n \cong \operatorname{cone}(\CoxComplex) \times \mbR,
\end{equation*}
where $\operatorname{cone}(\CoxComplex) = \big(\CoxComplex\times [0,\infty) \big)/ (x,0)\sim (y,0) \cong \mbR^{n-1}$. This decomposition allows one to describe each point $x\in \mbR^n$ via coordinates $x_\theta, \bar{x}, \lVert v_x \rVert$, where $x_\theta$ specifies a point on the Coxeter complex, $\lVert v_x \rVert$  is the ``cone parameter'' and $\bar{x}$ parametrises the remaining $\mbR$ (for details, see \cref{prop_CoxCoordsRn}, where the naming becomes clear as well). In summary, this describes $\mathcal{B}_n$ as a subset of
\begin{equation*}
    \mathcal{B}_n \subset \Sym \backslash \big(\operatorname{cone}(\CoxComplex) \times \mbR \times \operatorname{cone}(\CoxComplex) \times \mbR \big).
\end{equation*}
We call the coordinates that we obtain from this description \emph{Coxeter coordinates}.
It turns out that for each barcode, these coordinates are $b_\theta, \bar{b},\, \lVert v_b \rVert$ and $d_\theta, \bar{d},\, \lVert v_d \rVert$, where $\bar{b}$ and $\bar{d}$ are the averages of the births and deaths, $\lVert v_b \rVert$ and $\lVert v_d \rVert$ are their standard deviations and the coordinates $b_\theta$ and $d_\theta$ describe the permutation equivalence class of the barcode of \cite{TRN,trees_barcodesII}. The stratification one obtains is induced by the simplicial structure of $\CoxComplex$.

The advantages of these new coordinates are two-fold: Firstly, using points in Coxeter complexes, one obtains coordinates that uniquely specify barcodes and are yet compatible with the combinatorial structure of $\mathcal{B}_n$ given by permutation equivalence classes.
Secondly, one resolves the earlier-mentioned problem that permutation equivalence classes themselves carry no notion of ``size'': The decomposition of $\mathcal{B}_n$ into regions subdivides these equivalence classes by also taking into account the averages and standard deviations of births and deaths. This makes these regions a finer invariant than the permutation type.
Therefore, they offer a new way to study statistics of barcodes by using both the average and standard deviation of births and deaths, which are commonly used summaries in Topological Data Analysis (TDA), and permutation statistics tools. The latter include the number of descents for instance, or the inversion numbers, which have proven useful for the study of the inverse problem for trees and barcodes \cite{TRN,trees_barcodesII}.

\subsection{Related work}\label{related_work}

This paper is a follow-up of the work started in \cite{TRN, trees_barcodesII} to study the space of barcodes from a combinatorial point of view. It extends the approach of considering permutations to classify barcodes to a finer classification that also takes into account the average and standard deviation of births and deaths.
In \cite{master_thesis}, the author also observes a connection between barcodes and the symmetric group in a different setting, by studying the space of barcode bases using Schubert cells. Similarly, \cite{space_barcodes} also studies the space of barcode bases.

The idea of giving coordinates to the space of barcodes is not new \cite{diagram_to_vectors, trop_coord}. For example, the space of barcodes was given tropical coordinates in \cite{trop_coord}.
In \cite{alg_ring_barcodes}, it is mentioned that the space of barcodes can be identified with the $n$-fold symmetric product of $\mbR^2$, and the authors study the corresponding algebra of polynomials associated to the variety. 

Finally, defining a polyhedral structure on a space to study statistics has been done for spaces of (phylogenetic) trees \cite{BHV, grindstaff}. The connection between phylogenetic trees, merge trees and barcodes is studied in \cite{trees_barcodesII}. The polyhedral structure defined in this paper and in \cite{BHV} seem to be related, but we leave this as future work. 
   
\subsection{Overview}
In \cref{sec_background} we review the necessary background on barcodes and on Coxeter complexes. We use a standard way of realising $\Sym$ as a reflection group to explain what we mean with ``Coxeter coordinates'' on $\mbR^n$ in \cref{sec_cox_coords_Rn}.
We then describe the space $\mathcal B_n$ of barcodes with $n$ bars in terms of $\Sym \backslash \mbR^n \times \mbR^n$ in \cref{sec_alt_description}, before adapting the coordinates of $\mbR^n$ to $\mathcal B_n$ in \cref{sec_cox_coords_birthdeath}.   In \cref{sec_stratification}, we describe the stratification of $\mathcal{B}_n$ induced by these coordinates. \cref{cor_simplified_coordinates} decomposes the space of barcodes into regions indexed by the average and standard deviation of the births and deaths and the permutation associated to a barcode. 
Finally, in \cref{sec_metric_B_n}, we show that $\mathcal{B}_n$ can be given metrics inspired by the bottleneck and Wasserstein distances and that it defines an isometry between a subset of $\Sym \backslash \mbR^n \times \mbR^n$ and $\mathcal{B}_n$.
\section{Background}\label{sec_background}

\subsection{Background on TDA}
We start by reviewing the necessary background on TDA. For the reader who is completely new to this, we refer to the reviews \cite{topdata,perssurvey,barcodes}. Even though this work focuses on the space of barcodes and could be apprehended from a purely combinatorial point of view, we shortly mention where barcodes arise in the field of TDA. This section is not necessary for the understanding of this paper, and we will give the combinatorial definition of barcodes that we use in the next section.

 Barcodes are topological summaries of a \emph{filtered topological space}, i.e.~a sequence of spaces ordered by inclusion. To obtain a barcode from a filtered space, one computes homology at each step and considers the maps induced by the inclusions. The output is called a \emph{persistence module}, and it summarises the evolution of the homology at each step of the filtration.

More precisely, let $\{X_t\}_{ t \in \mbR}$ be a filtered topological space, that is, each $X_t$ is a topological space and $X_t \subseteq X_{t'}$ if $t\leq t'$. The $k$-th persistence module associated to $\{X_t\}_{ t \in \mbR}$ is given by $\mathbb{H}_k(\{X_t \}_{ t \in \mbR})$, where $\mathbb{H}_k$ denotes the $k$-th homology functor (over a field $\Bbbk$). The Crawley--Bovey Theorem \cite{crawley2015decomposition} states that under mild tameness conditions on $\{X_t\}_{ t \in \mbR}$, the associated persistence module can be decomposed as a direct sum of interval modules $\bigoplus _{j\in \mathcal J} \Bbbk_{I_j}^{\oplus n_j}$, where the \emph{interval module} $\Bbbk_{I_j}$ is the free $\Bbbk$-module of rank $1$ on the interval $I_j\subseteq \mbR$, with identity maps internal to $I_j$, and is $0$ elsewhere. This decomposition is unique up to reordering. Each interval represents the lifetime of a cycle in the filtered space. For instance, if a $1$-cycle (a loop) appears in the topological space $X_{b_j}$ for the first time and becomes a boundary (gets ``filled in'') in $X_{d_j}$, then this $1$-cycle will be represented by the interval $I_j=  [b_j,d_j)$.  The \emph{barcode} associated to the persistence module is the multiset
$$B = \{I_j\}_{j \in \mathcal{J}},$$ where each interval $I_j$ appears $n_j$ times. 
Usually, each $I_j$ is a half open interval $I_j=[b_j,d_j)$, where $b_j$ is called the \emph{birth} of the homological feature corresponding to $I_j$ and $d_j$ is called its \emph{death}. If the interval $I_j$ is a half infinite interval, i.e.~it is of the form $[b_i,\infty)$, it is called an \emph{essential class}.

In this paper, we will identify such an interval with the pair $(b_j,d_j)$, since we are mostly interested in the combinatorics of the pairs and not the corresponding persistence module. Moreover, $b_j$ and $d_j$ will always take finite values in $\mbR$.

\subsubsection{The space of barcodes}
We introduce here the main definitions used in this paper. We start by a more combinatorial definition of barcodes that we will use in this article.

\begin{definition}
A \emph{barcode} $\{(b_i,d_i)\}_{i \in J}$ is a multiset of pairs $(b_i,d_i)\in \mbR^2$ such that $ b_i < d_i $ for each $i \in J$ and $|J|<\infty$. Each such pair is called a \emph{bar}; its first coordinate $b_i$ is called the \emph{birth} (time) and the second one $d_i$ is called its \emph{death} (time). A barcode is called \emph{strict} if $b_i \neq b_j$ and $d_i \neq d_j$ for $i \neq j$.
We let $\mathcal{B}_n$ denote the set of barcodes with $n$ bars and $\mathcal{B}_n^{st}$ the set of strict barcodes with $n$ bars. 
\end{definition}

 \begin{remark}
The reader familiar with persistent homology will notice that we suppose that the bars corresponding to essential classes have finite values instead of being half-open intervals. This is usually the case in practical applications, where such essential classes are given finite values for representing them on a computer. We also assume that every barcode consists of only finitely many bars.
\end{remark}

\begin{remark}
The definition of strict barcodes was first introduced in \cite{TRN} to define the bijection between the symmetric group on $n$ elements and some equivalence classes of barcodes that we introduce in the next section.
The setting in this paper is slightly different from \cite{TRN} and \cite{trees_barcodesII}, because all the barcodes considered there are specific to merge trees and arise from their $0$-th persistent homology. This is why the definition of a strict barcode in \cite{TRN} and \cite{trees_barcodesII} assumes the existence of an essential bar $(b_0,d_0)$ that contains all the others. In this paper however, barcodes can come from arbitrary filtrations in arbitrary dimension, and such a bar $(b_0,d_0)$ need not exist. Therefore we slightly adapt the definition of a strict barcode and the relation to the symmetric group in the next sections. 
\end{remark}

In practice, for finite barcodes, the indexing set $J$ is commonly the set $\{1,...,n\}$,  giving the bars in the barcode an arbitrary but fixed ordering. We will also adopt this convention from now on. Note however that reordering the bars might change the indexing, but not the underlying barcode (see \cref{ex_indexing_barcode}). It can sometimes be convenient to assume that the indexing is such that the births are ordered increasingly $b_1<b_2<...<b_n$, but we do not make this assumption in this paper unless specified.
 
We often represent a barcode by the set of intervals $[b_i,d_i] \subset \mbR$ (as in \cref{fig_barcode}). Another common way to represent barcodes is  what is called a \emph{persistence diagram}, where the pairs $(b_i,d_i)$ are represented as points in $\mbR^2$ (as in \cref{fig_bottleneck}). These points lie above the diagonal since $b_i < d_i$ for all $i$.
 
 \begin{example}
 \label{ex_indexing_barcode}
\cref{fig_barcode} shows an example of a strict barcode with two different indexing conventions. 
\begin{figure}[H]
    \centering
    \includegraphics[scale=0.4]{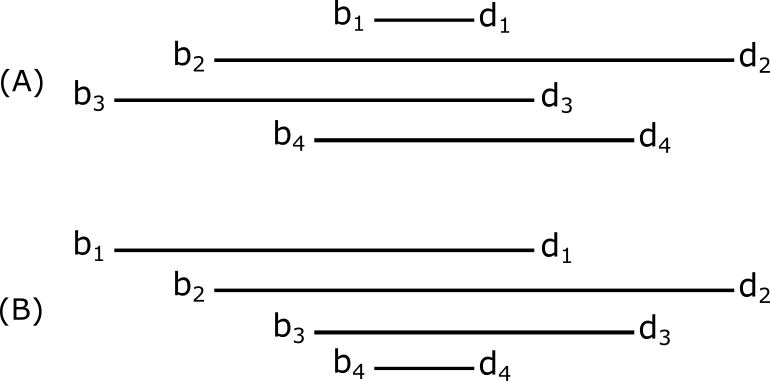}
    \caption{(A) A barcode with $4$ bars. (B) The same barcode with a different indexing  where the bars are ordered by increasing birth times.}
    \label{fig_barcode}
\end{figure}

 \end{example}

To turn the set of barcodes into a topological space, one needs to specify a topology. One option to do this is by introducing the bottleneck or Wasserstein distances, two commonly used metrics for barcodes. 
Intuitively, the bottleneck distance between two barcodes $B$ and $B'$ tries all possible matchings between the bars of $B$ and the bars of $B'$ and chooses the one that minimises the ``energy'' required to move the matched pair of bars with maximal separation. However, it does not only consider matching of bars between $B$ and $B'$ but also with points on the diagonal $\Delta = \{(x,x) \mid x \in \mbR\}$.

\begin{definition}\label{bottleneck}
Let $B= \{(b_i,d_i)\}_{i \in \{1,...,n\}} $ and $B'= \{(b'_i,d'_i)\}_{i \in \{1,...,m\}}$ be two barcodes. The \emph{bottleneck distance} between $B$ and $B'$ is 
$$d_B(B,B') = \min_{\gamma} \max_{x \in B} \lVert x -\gamma(x) \rVert_\infty,$$ where $\gamma$ runs over all possible matchings, i.e.~maps that assign to each bar $(b_i,d_i) \in B$ either a bar in $B'$ or a point in the diagonal $\Delta$, such that no point of $B'$ is in the image more than once. Here, $\lVert \cdot \rVert_\infty$ is the $l^\infty$-norm on $\mbR^2$.
\end{definition}

\begin{remark}
The permutation $\gamma$ acts as a ``reindexing'' of the indices of $B$ and $B'$, and in particular ensures that $d_B(B,B')$ does not depend on any indexing of the bars.
\end{remark}

The \emph{Wasserstein distance} is defined in a similar way by taking the sum over all $l_2$-distances between $x$ and $\gamma(x)$ instead: 
$$d_W(B,B') = \min_{\gamma} \sqrt{ \Sigma_ {x \in B}\lVert x-\gamma(x) \rVert_2^2)}.$$ 

\begin{remark}
Note that in general, the barcodes $B$ and $B'$ need not have the same number of bars. The diagonal allows matchings between barcodes with different number of bars, since ``ummatched'' bars can be sent to the diagonal. In this paper however, we are study the set of barcodes $\mathcal{B}_n$ with exactly $n$ bars (for arbitrary, but fixed $n$) and restrict ourselves to this case.

We are mainly interested in $\mathcal{B}_n$ as a set and the main results we prove do not depend on the metric that is chosen on $\mathcal{B}_n$. We will still with a slight abuse of notation mostly talk of $\mathcal{B}_n$ as a \emph{space}, without specifying a specific metric on it. An exception to that is \cref{sec_metric_B_n}, where we explain how a metric $\tilde{d}_B$ on $\mathcal{B}_n$, which is closely related to the bottleneck distance, occurs in an alternative description of the \emph{set} $\mathcal{B}_n$ that we work with later on. 
\end{remark}

\subsubsection{Relation to the symmetric group}
We write $\Sym$ for the symmetric group on $n$ letters, i.e.~the group of all permutations of $\{1,\ldots,n\}$. 
We usually use the one-line notation for permutations. That is, we specify $\sigma\in \Sym$ by the its image of the ordered set $\{1,\ldots, n\}$, e.g. we write $\sigma=[132] \in \Sym[3]$ if $\sigma(1)=1$, $\sigma(2)=3$ and $\sigma(3)=2$.
We make an exception for transpositions to simplify the notation: the transposition that switches $i$ and $j$ is denoted by $(i,j)$.

\begin{definition}\label{def_permutation} \cite{TRN}
Let $B=\{(b_i,d_i)\}_{i \in \{1,...,n\}}\in \mathcal{B}_n^{st}$ be a strict barcode. If we order the births increasingly such that $b_{i_1} < \ldots< b_{i_n}$, the indexing in $\{1,...,n\}$ gives a permutation $\tau_b$ by $\tau_b(k)=i_k$, i.e.~$\tau_b$ is the (unique) permutation such that
\begin{equation}\label{taub}
    b_{\tau_b(1)}< \ldots< b_{\tau_b(n)}.
\end{equation}
Similarly, ordering the deaths $d_{j_1}<\ldots < d_{j_n}$ gives rise to a permutation $\tau_d$ with $\tau_d(k)=j_k$. The \emph{permutation $\sigma_B$ associated to $B$} is defined as $\sigma_B=\tau_b^{-1}\tau_d$; it tracks the ordering of the death values with respect to the birth values.
\end{definition}

\begin{remark}
The permutations $\tau_b$ and $\tau_d$ both depend on the indexing choice of the $b_i$ and $d_i$. However, the permutation $\sigma$ does not depend on any indexing of the births and deaths, it is intrinsic to the multiset $B$. Indeed, $\sigma_B$ can be defined directly as the permutation that sends the $i$-th death (in increasing order) to the $\sigma(i)$-th birth (idem). 
If we assume that the births are ordered increasingly, then $\tau_b=\id$ and $\sigma_B$ can be defined directly by $\sigma_B=[j_1 j_2 \ldots j_n]$, the indices of the deaths when they are ordered increasingly.

\end{remark}

\begin{example}
\cref{fig_barcode}A shows an example of a strict barcode. Its birth permutation is $\tau_b=[3241]$, since $$  b_3<b_2<b_4<b_1.$$
Similarly, its death permutation is $\tau_d =[1342]$, since $d_1<d_3<d_4<d_2$. The permutation $\sigma_B$ associated to the barcode of \cref{fig_barcode}A is $\sigma_B = [4132] = \tau_b^{-1}\tau_d$.
\cref{fig_barcode}B shows the same barcode with the bars ordered by birth times. The corresponding permutations $\tau_b=[1234]$ and $\tau_d=[4132]$ are different, but the product $\sigma_B = \tau_b^{-1}\tau_d = [4132]$ is the same, as it does not depend on the indexing of the bars. 
Further examples are depicted in \cref{fig_cayley_s4}.
\end{example}

\begin{figure}
    \centering
    \includegraphics[width = .85 \textwidth]{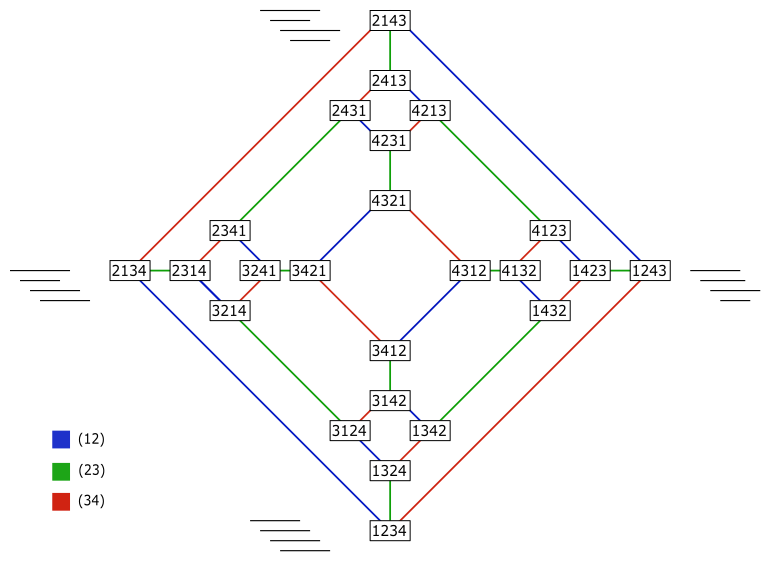}
    \caption{(Figure from \cite{TRN}) The Cayley graph of $\Sym[4]$ generated by the three transpositions $(12),(23),(34)$. Four barcodes are drawn next to the extremities of the graphs (permutations $[1234],[2134],[2143],[1243]$) to illustrate a typical barcode corresponding to each permutation.}
    \label{fig_cayley_s4}
\end{figure}

We extend \cref{def_permutation} to non-strict barcodes in \cref{sec_stratification}.

\subsection{Background on Coxeter groups and complexes}

\label{sec_cox_background}

\paragraph{Coxeter groups}
 \emph{Coxeter groups} form a family of groups that was defined by Tits in its modern form. They are abstract versions of reflection groups; in fact, the family of finite Coxeter groups coincides with the family of finite reflection groups. Besides their close connections to geometry and topology \cite{Dav:geometrytopologyCoxeter}, Coxeter groups have a rich combinatorial theory \cite{BB:CombinatoricsCoxetergroups}. They appear in many areas of mathematics, e.g.~as Weyl groups in Lie theory.
We will view $\Sym$ as one of the most basic examples of a Coxeter group.

Usually, one does not consider a Coxeter group $W$ by itself but instead a \emph{Coxeter system} $(W,S)$, where $S$ is a generating set of $W$ that consists of involutions called the \emph{simple reflections}. In what follows, we will tacitly assume that such a set of simple reflections is always fixed when we talk about a Coxeter group $W$. In the case where $W=\Sym$, we will take $S$ to be the set of adjacent transpositions $S = \ls (i,i+1) \mid 1\leq i \leq n-1 \rs$.
A  rank-$(|S|-1-k)$  (standard) \emph{parabolic subgroup} of $W$ is a subgroup of the form $P_T = \ll T \rr$, where $T\subset S$ is a subset of size $(|S|-1-k)$.

\paragraph{Coxeter complexes}
Each Coxeter group $W$ can be assigned a simplicial complex $\CoxComplex[W]$, the \emph{Coxeter complex}, that is equipped with an action of $W$.
If $W$ is a finite group with set of simple reflections $S$, the complex $\CoxComplex[W]$ is a triangulation of a sphere of dimension $|S|-1$. Coxeter complexes have nice combinatorial properties and are in particular colourable flag complexes \cite[Section 1.6]{AB:Buildings} that are shellable \cite{Bjo:Somecombinatorialalgebraic}.

The top-dimensional simplices of $\CoxComplex[W]$ are in one-to-one correspondence with the elements of the group $W$. Furthermore, one
recovers the Cayley graph of $(W,S)$ as the \emph{chamber graph} of $\CoxComplex[W]$, i.e.~the graph that has a vertex for each top-dimensional simplex of $\CoxComplex[W]$ and an edge connecting two vertices if the corresponding simplices share a codimension-1 face \cite[Corollary 1.75]{AB:Buildings}.

More generally, the set of $k$-simplices in $\CoxComplex[W]$ is in one-to-one correspondence with the cosets of rank-$(|S|-1-k)$ parabolic subgroups of $W$:
\begin{definition} 
\label{def_cox_complex}
The \emph{Coxeter complex} $\CoxComplex[W]$ of the Coxeter system $(W,S)$ is defined as the simplicial complex
\[\CoxComplex[W] = \bigcup_{T \subseteq S}  W / P_T =  \{ \tau P_T  \mid \tau \in W, T \subseteq S\}, \]
where each simplex $\tau P_T$ has dimension\footnote{Note that we take the (combinatorial) convention that this simplicial complex has a unique face of dimension $-1$. This face does not appear in the geometric realisation.} $\dim (\tau P_T)=|S\setminus T|-1$ and the face relation is defined by the partial order 
\begin{equation}
\label{eq_poset_CoxComplex}
    \tau P_T \leq \tau' P_{T'} \Leftrightarrow \tau P_T \supseteq \tau'P_{T'}.
\end{equation} 
The group $W$ acts simplicially on  $\CoxComplex[W]$ by left multiplication on the cosets, $\gamma \cdot (\tau P) \coloneqq \gamma \tau P$.
\end{definition}

\begin{remark}
With a slight abuse of notation, we will in what follows often use the cosets $\tau P$ to also denote simplices in the geometric realisation of the Coxeter complex. To be coherent with the definition of a stratification (\cref{def_strat}), we will always consider these simplices to be closed.
\end{remark}

\paragraph{The Coxeter complex $\boldsymbol{\CoxComplex}$}
\begin{figure}
    \centering
    \includegraphics[width = .8
    \textwidth]{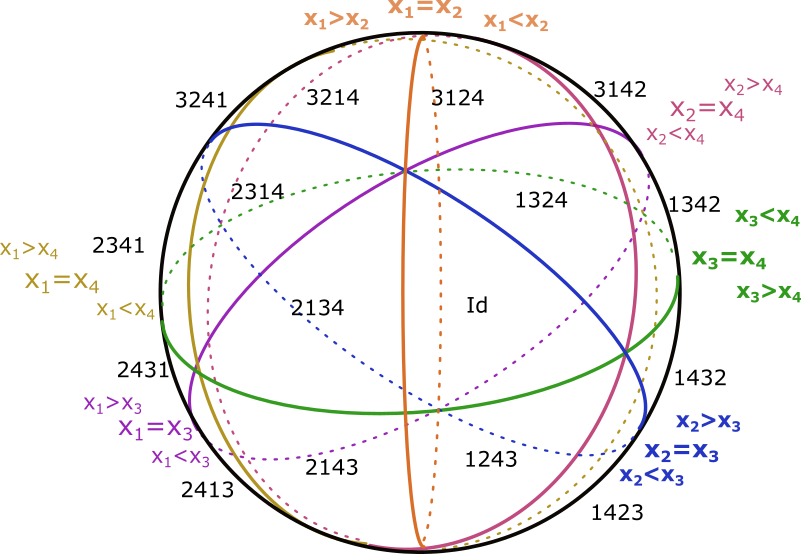}
    \caption{The geometric realisation of the Coxeter complex $\CoxComplex[{\Sym[4]}]$. The permutation corresponding to each triangle of the front of the sphere is indicated in black. The hyperplanes $x_i = x_j$ depicted in colours correspond to the transpositions $(i,j)$. The hyperplanes corresponding to adjacent transpositions $(i,i+1)$ are in boldface. A detailed description of how to obtain such a geometric realisation of the Coxeter complex can be found in \cref{sec_cox_coords_Rn}.}
    \label{coxeter_4}
\end{figure}
For the case $W = \Sym[n]$ that we are interested in, the Coxeter complex $\CoxComplex$ is of dimension $n-2$ and is isomorphic to the barycentric subdivision of the boundary of an $(n-1)$-simplex. It can be realised geometrically as a triangulation of the $(n-2)$-sphere.  This complex is the dual to the permutohedron of order $n$ (see \cref{fig_permutohedron_dual}). \cref{coxeter_4} depicts the Coxeter complex $\CoxComplex[{\Sym[4]}]$. 
The top-dimensional simplices of $\CoxComplex[{\Sym}]$ are in one-to-one correspondence with the elements of $\Sym$. Two such simplices share a codimension-1 face if and only if the corresponding permutations differ by precomposing with an adjacent transposition $(i,i+1)$, i.e.~by exchanging two neighbouring entries of the permutation.
As a consequence, if $x$ lies in the interior of a maximal simplex of the geometric realisation of $\CoxComplex$, it can be assigned a permutation $\tau \in \Sym$.
If $x$ lies on a face of dimension $k$, then $\tau$ is well-defined only up to applying an element of a parabolic subgroup $P\leq \Sym$ that is generated by $|S|-1-k = n-2-k$ adjacent transpositions. 
A concrete embedding of $\CoxComplex$ in $\mbR^n$ will be described in more detail in  \cref{sec_cox_coords_Rn}.

For later reference, we note that the identification $S^{n-2} \cong \CoxComplex$ gives a \emph{stratification} of the sphere by its simplicial decomposition. The strata are the (closed) simplices of the geometric realisation and the stratification is over the partially ordered set (poset) specified by \cref{eq_poset_CoxComplex}.

\begin{definition} \cite{haefliger} \label{def_strat}
A set $X$ is \emph{stratified} over a poset $\mathcal P$ if there exists a collection of subsets $\{X_i\}_{i \in \mathcal P}$ of $X$ such that:

\begin{enumerate}
    \item $X = \bigcup_i X_i$;
    \item $i\leq j $ if and only if $X_i \subseteq X_j$;
    \item if $X_i \cap X_j \neq \emptyset,$ then it is a union of strata;
    \item For every $x \in X$, there exists a unique $i_x \in \mathcal P$ such that $\bigcap_{X_i \ni x} X_i = X_{i_x}$.
\end{enumerate}
Each $X_i$ is called a \emph{stratum}.
\end{definition}

\section{Coxeter complex coordinates on $\mbR^n$}
\label{sec_cox_coords_Rn}
In this section, we describe $\mbR^n$ as the product of a cone over the Coxeter complex $\CoxComplex$  with a $1$-dimensional space orthogonal to it. This description is obtained by describing a standard way for realising $\Sym$ as a reflection group \cite[Example 1.11]{AB:Buildings}. In terms of Coxeter groups, this is often called the ``dual representation'', see e.g.~\cite[Section 2.5.2]{AB:Buildings}.
\cref{example_coxeter_S3} below goes through the following steps in detail for the case $n=3$.

In what follows, we will consider $\mbR^n$ with the $l^2$-norm $\lVert \cdot \rVert$ that is induced by the standard scalar product $\ll \cdot , \cdot \rr$. 
We let $e_1,\ldots, e_n$ denote the standard basis.
The symmetric group $\Sym$ acts on $\mbR^n$ by permuting this standard basis.
This action can be expressed in coordinates as
\begin{equation}\label{eq_action}
     \gamma \cdot (x_1,\ldots, x_n) = (x_{\gamma^{-1}(1)},\ldots, x_{\gamma^{-1}(n)}).
\end{equation}
It is norm-preserving and fixes the $1$-dimensional subspace $L=\langle e \rangle$ spanned by $e \coloneqq e_1 +\cdots + e_n = (1,\ldots,1)$.
Hence, there is an induced action on the orthogonal complement $V = e^\perp$, which can be described as
\begin{equation*}
    V = \left\{(x_1,\ldots, x_n) \in \mbR^n \, \middle|\, \Sigma_{i=1}^n x_i = 0 \right\}.
\end{equation*}
Note that $L$ is the subspace consisting of all $(x_1, \ldots, x_n )\in \mbR^n$ where $x_i = x_j$ for all $i,j$. So in particular, every $(x_1, \ldots, x_n )\in \mbR^n\setminus L$ has at least two coordinates that are different from one another.

The subspace $V$ has a natural structure of a cone over the Coxeter complex $\CoxComplex$ associated to $\Sym$, see \cref{rem_cone}.
The transposition $(i,j)\in \Sym$ acts on $V$ by orthogonal reflection along the hyperplane 
\begin{equation*}
    \left\{(x_1,\ldots, x_n) \in \mbR^n \, \middle|\, x_i = x_j \right\},
\end{equation*}
permuting the $i$-th and $j$-th coordinates.
Let $\mcH$ be the collection of all these hyperplanes, and let $S_r$ denote the $(n-2)$-sphere of radius $r>0$ around the origin in $V$ (with respect to the norm induced by the restriction of the standard scalar product on $\mbR^n$), i.e.~$S_r = \{ v \in V \mid \Vert v \Vert = r\}$. 

\begin{lemma}[{\cite[Examples 1.10, 1.4.7 \& 1.81]{AB:Buildings}}]\label{lem:sphere_coxcom}
The hyperplanes $\mcH$ induce a triangulation of $S_r$. The resulting simplicial complex $\Sigma$ is isomorphic to the Coxeter complex $\CoxComplex$ as $\Sym$-spaces.
\end{lemma}

The set of points $x\in \mbR^n$ such that all coordinates are different is the \emph{configuration space} 
\[\textup{Conf}_n(\mbR)= \{ (x_1,\ldots, x_n) \in \mbR^n \mid i \neq j \implies x_i \neq x_j\}. \]
The previous lemma describes how a permutation in $\Sym$ can be associated to each point $x\in \textup{Conf}_n(\mbR).$
To understand why this is true, observe that if $C$ is a connected component of $S_r\backslash \bigcup \mcH$, then for all $(x_1,\ldots, x_n)\in C$:
\begin{itemize}
    \item if $i\not= j$, then $x_i\not= x_j$, i.e.~ $(x_1,\ldots, x_n)\in \textup{Conf}_n(\mbR)$;
    \item \label{order_connecte_comps} if $(y_1,\ldots, y_n)\in C$, then $y_i < y_j$ if and only if $x_i < x_j$.
\end{itemize}
In particular, there is a unique $\tau\in \Sym$ such that 
\begin{equation}
    \label{eq_conn_comp}
    (x_1,\ldots, x_n)\in C \iff x_{\tau(1)}< x_{\tau(2)} < \cdots < x_{\tau(n)}.
\end{equation}
In other words, the order of the elements $x_1,\ldots, x_n$ is given by $\tau( (1, \ldots, n) )$, see \cref{coxeter_4} above for the case $n=4$.
The connected components of $S_r\backslash \bigcup \mcH$ are exactly the (interiors of) the maximal simplices of $\Sigma$. Sending each such component $C$ to the facet of $\CoxComplex$ that corresponds to the permutation $\tau$ defined by \cref{eq_conn_comp} gives the desired isomorphism $\Sigma\cong \CoxComplex$. 
\newline

Using spherical coordinates, we can express every point $v\in V$ in terms of a radial component $r> 0$ and an angular component, which is equivalent to specifying a point $v_{\theta}\in S_r$ (i.e.~a point in the geometric realisation of \CoxComplex). 
The upshot of this is that we obtain a new set of coordinates for points in $\mbR^n\setminus L$.
\begin{proposition}
\label{prop_CoxCoordsRn}
Let $n \geq 2$.
There exist two projection maps \[p:\mbR^n \longrightarrow \mbR \times \mbR_{\geq 0}: x \mapsto ,(\bar{x}, \lVert v_x \rVert), \] where $\bar{x}=\frac{1}{n}\sum_{i=1}^n x_i$ and $\lVert v_x \rVert= \left( \sum_{i=1}^n |x_i-\bar{x}|^2\right)^{1/2}$, and
\[q:\mbR^n\setminus L \longrightarrow \CoxComplex\] that define a bijection 
\begin{equation*}
    (\restr{p}{\mbR^n \setminus L},q):\mbR^n \setminus L \longrightarrow \mbR \times \mbR_{>0} \times \CoxComplex.
\end{equation*}
Let $\Sym$ act on $\mbR^n$ by permuting the coordinates (\cref{eq_action}) and on the product $\mbR \times \mbR_{>0} \times \CoxComplex$ by extending the action on $\CoxComplex$ trivially on the first two factors. Then the map $(\restr{p}{\mbR^n \setminus L},q)$ is $\Sym$-equivariant.
\end{proposition}
\begin{proof}
For every $x \in \mbR^n$, the orthogonal decomposition $\mbR^n = \ll e \rr \oplus V$ gives a unique way to write $x = \bar{x}\cdot e + v_x$ with $\bar{x}\in \mbR$ and $v_x\in V$, where
\begin{equation*}
    \bar{x} = \frac{\ll e, x \rr}{\ll e,e \rr} = \sum_{i=1}^n x_i/n = \frac{1}{n}\sum_{i=1}^n x_i.
\end{equation*}
We can describe the projection $v_x = x-\bar{x}\cdot e \in V$ in spherical coordinates. Its norm (the radius of the sphere) is 
\begin{equation*}
    \lVert v_x \rVert = \lVert x-\bar{x}\cdot e \rVert = \left( \sum_{i=1}^n |x_i-\bar{x}|^2\right)^{1/2}, 
\end{equation*}
so $v_x$ is determined by this value together with a point $x_\theta$ on the $(n-2)$-sphere $S_{\lVert v_x \rVert }$, or equivalently on the geometric realisation of $\CoxComplex$. 
Notice that $x\in L$ if and only if $v_x=0$, as the line $L$ intersects $V$ at its origin.

\medskip

We define the map $p: \mbR^n  \longrightarrow \mbR \times \mbR_{\geq0}: x \mapsto (\bar{x}, \lVert v_x \rVert)$ and the map $q: \mbR^n \setminus L \longrightarrow S^{n-2}: x \mapsto x_\theta.$
The point $x_\theta$ is well-defined since $x \notin L$ and therefore there exist $i,j$ such that $x_i \neq x_j$.
It is easy to see that $(\restr{p}{\mbR^n \setminus L},q)$ is a bijection, i.e.~that given $c_1\in \mbR$, $c_2\in \mbR_{> 0}$ and $c_3\in \CoxComplex$, there is a unique $x\in \mbR^n \setminus L$ such that $c_1 = \bar{x}$, $c_2 = \lVert v_x \rVert$ and $c_3 = x_\theta$.

\medskip

The fact that $(\restr{p}{\mbR^n \setminus L},q)$ is $\Sym$-equivariant follows from \cref{lem:sphere_coxcom} and because permuting the coordinates of $x\in \mbR^n$ changes neither the average $\frac{1}{n}\sum_{i} x_i$ nor the standard deviation $\left( \sum_{i} |x_i-\bar{x}|^2\right)^{1/2}$.
\end{proof}

To summarise, every point  $x = (x_1,\ldots, x_n)\in \mbR^n \setminus L$ determines the following three things:
\begin{enumerate}
    \item its projection to $L$, given by $\bar{x} = \frac{1}{n}\sum_{i=1}^n x_i \in \mbR$;
    \item the norm of its projection to $V$, given by $\lVert v_x \rVert =\left( \sum_{i=1}^n |x_i-\bar{x}|^2\right)^{1/2} \in\nolinebreak \mbR_{> 0}$;
    \item a point $x_\theta$ in the geometric realisation of the Coxeter complex $\CoxComplex$ associated to $\Sym$.  \label{Coxeter_coord}
\end{enumerate}
Furthermore, $x$ is uniquely determined by these three coordinates.

\begin{remark}\label{rem_cone}
There is an isomorphism $\mbR_{> 0}\times \CoxComplex \cong \operatorname{cone}(\CoxComplex) \setminus \ls \ast \rs$, where 
\begin{equation*}
    \operatorname{cone}(\CoxComplex) = \CoxComplex \times [0,\infty) / (x,0)\sim (y,0).
\end{equation*}
Since $\mbR^n = \mbR^n \setminus L \sqcup L $, the above map $(\restr{p}{\mbR^n \setminus L},q)$ gives rise to a decomposition $\mbR^n \cong \operatorname{cone}(\CoxComplex) \times \mbR$. Indeed, the line $L \subset \mbR^n$ corresponds to points $x \in \mbR^n$ with $v_x=0$, which could be seen as ``spheres of radius $0$'' in the projection $q$.
\end{remark}

\label{record_regions_Rn}

\begin{example}
\label{example_coxeter_S3} 
We go through the previous construction in detail for the case of $\mbR^3$ equipped with the natural action of the symmetric group $\Sym[3]$, illustrating the example in \cref{coxeter_S3}. Consider $\mbR^3=\ll e_1,e_2,e_3 \rr$. The symmetric group $\Sym[3]$ acts on it by permuting the coordinates of each vector $(x_1,x_2,x_3)$:
\begin{equation*}
     \gamma \cdot (x_1,x_2, x_3) = (x_{\gamma^{-1}(1)},x_{\gamma^{-1}(2)}, x_{\gamma^{-1}(3)}).
\end{equation*}
Each $\gamma \in \Sym[3]$ can be written as a product of transpositions $(i,j)$ and its action on $\mbR^3$ is given by the performing the corresponding sequence of reflections along the hyperplanes $x_i=x_j$. The three (2-dimensional) planes corresponding to the equations $x_1=x_2$, $x_2=x_3$ and $x_1=x_3$ are indicated as lines on the left hand side of \cref{coxeter_S3} to make the picture clearer. 
The subspace $L$ that is invariant under this action is spanned by the vector $(1,1,1)=e$, shown in red in \cref{coxeter_S3}.

\begin{figure}
    \centering
    \includegraphics[width=\textwidth]{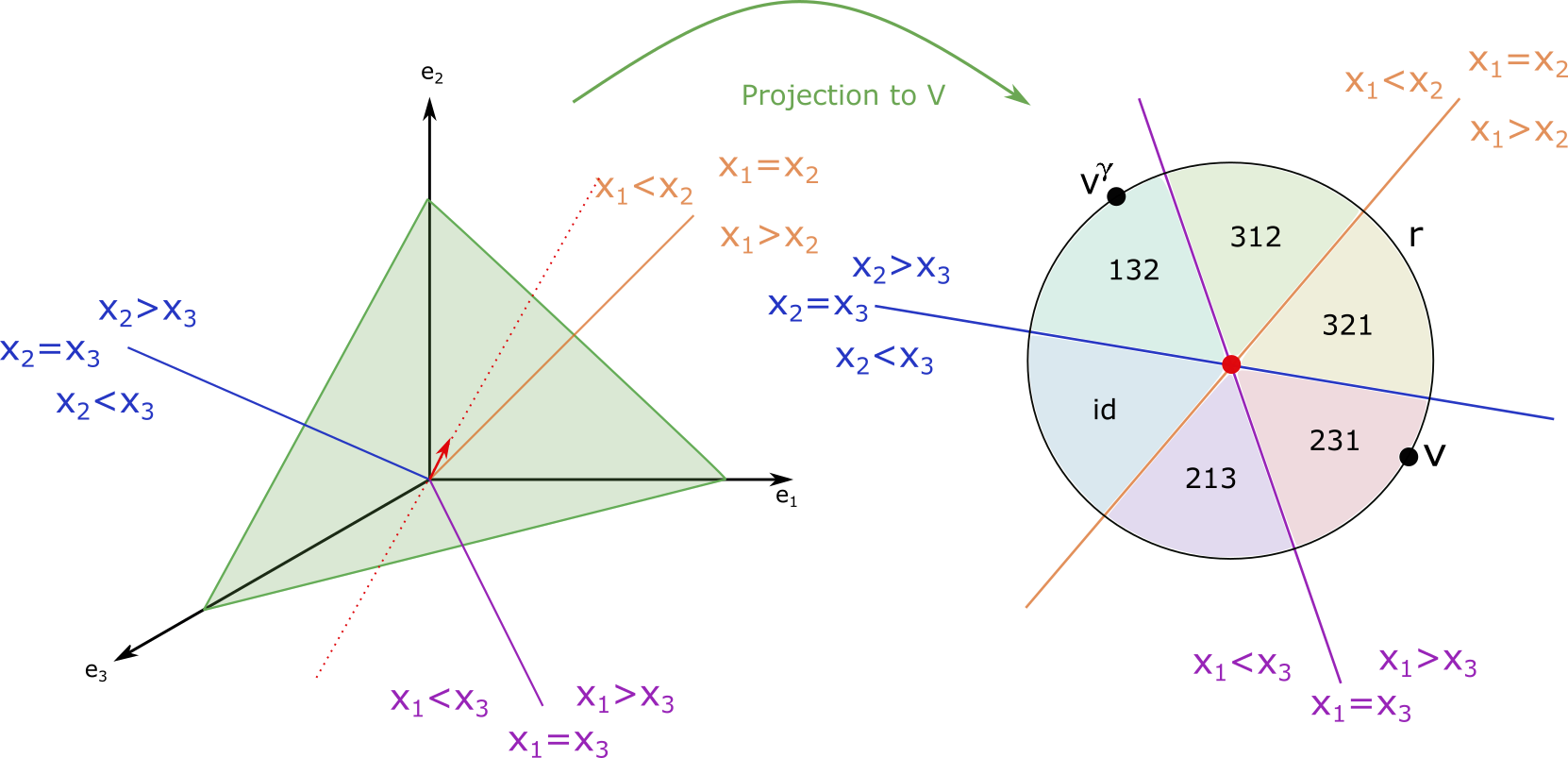}
    \caption{Example of the decomposition of $\mbR^3$ in Coxeter coordinates. }
    \label{coxeter_S3}
\end{figure}

We can define new coordinates on $\mbR^3$, lying in $\ll e \rr = L$ and $e^{\perp} =V$, a $2$-dimensional subspace whose affine shift is depicted in green in \cref{coxeter_S3}, reflecting the decomposition of $\mbR^3$ into a product of $\ll e \rr$ and $V$. A point $x \in \mbR^3$ can now be written as $\bar{x} \cdot e + v_x$, where $\bar{x} \in \mbR$ and $v_x \in V$.
\\
We show on the right hand side of \cref{coxeter_S3} how $V$, represented as $\mbR^2$, has the structure of a cone over a Coxeter complex. 
 The figure shows the projections of the planes $x_1=x_2$, $x_2=x_3$ and $x_1=x_3$ and the intersection of $V$ with the subspace $\ll e\rr $ (red dot). 
To obtain the cone structure on $V$, we give it spherical coordinates (i.e.~polar coordinates in this case).
The first coordinate is 
the radius $r$, which determines a $1$-sphere centred at the origin (the black circle). 
On the circle, a point $v_x$ is determined by an angle $x_\theta$. 
Intersecting the circle with the hyperplanes, we decompose it into $|\Sym[3]|=6$ (coloured) strata indexed by the symmetric group and forget about the angle $x_\theta$. For instance, if $v = (v_1,v_2,v_3)$ with $v_2<v_3<v_1$, the point $v$ lies in the stratum indexed by $[231]$; this is the unique region that lies on those sides of the hyperplanes that satisfy $x_1>x_2$, $x_2<x_3$ and $x_1>x_3$.

Let $\gamma=(12)$. It acts on $v$ via $\gamma \cdot v = (v_{\gamma^{-1}(1)},v_{\gamma^{-1}(2)},v_{\gamma^{-1}(3)}) = (v_2,v_1,v_3)$. We denote its image by $  v^\gamma \coloneqq \gamma \cdot v$. The order of the coordinates of $v^\gamma$ satisfies
$v_1^\gamma \leq v_3^\gamma \leq v_2^\gamma$, so $v^\gamma$ lies in the stratum indexed by the permutation $[132]$. The image $v^\gamma$ of $v$ through the action of $\gamma$ corresponds to the reflection through the hyperplane $x_1=x_2$.
\end{example}

\begin{remark}\label{rem_uniquely_det}
There are two special cases in \cref{prop_CoxCoordsRn}, when $x_i = x_j$ for all $i,j$, i.e.~$(x_1, \ldots, x_n)\in L$ and when $x_i \not= x_j$ for all $i\not= j$, i.e.~$(x_1, \ldots, x_n)\in \textup{Conf}_n(\mbR)$. For the former, we have $p(x)=(\bar{x},\lVert v_x \rVert)=(x_i,0) $ and $x_\theta$ is not defined. For the latter, 
$q(x)=x_\theta$ lies in the interior of a top-dimensional simplex of $\CoxComplex$. Hence, it determines a \emph{unique} element $\tau_x \in \Sym$.
In fact, these are just the two extremes of a family of situations that can occur:

If $x_i=x_j$ for some $i \neq j$, then $x_\theta$ lies on the corresponding hyperplane in $\mcH$ and hence on a lower-dimensional face of $\CoxComplex$. There exists a permutation $\tau \in \Sym$ such that \begin{equation*}
    x_{\tau(1)}\leq x_{\tau(2)} \leq \cdots \leq x_{\tau(n)},
\end{equation*} but $\tau$ is not unique. It is defined only up to multiplication by the subgroup 
\begin{equation*}
    P = \ls \gamma \in \Sym \,\middle| \, x_{\tau(i)} = x_{\tau \gamma(i)} \rs.
\end{equation*}
Note that $P$ is generated by adjacent transpositions $(i,i+1)$,
i.e.~it is of the form $\ll T \rr$, where $T\subset S$ is a subset of the set $S$ of simple reflections of $\Sym$. Hence, it is a parabolic subgroup of $\Sym$ (see \cref{sec_cox_background}).
The number of adjacent transpositions in $P$ depends on how many coordinates of $(x_1, \ldots, x_n)$ agree, or, equivalently, the number of hyperplanes in $\mcH$ it lies on. Intuitively speaking, one could phrase this as ``the more of the $x_i$'s take the same value, the less `permutation information' is left''.
The coset 
\begin{equation*}
        \tau P = \{ \rho \in \Sym \mid x_{\rho(1)} \leq \ldots \leq x_{\rho(n)}\},
\end{equation*}
corresponds to the lowest dimensional face of $\CoxComplex$ that $x$ lies on. It depends only on the values of the $x_i$, not on the choice of $\tau$. If $x\in L$, we have $\tau P=\Sym$. This could be interpreted as the degenerate case where $x_\theta$ lies on the unique $(-1)$-dimensional face of $\CoxComplex$ (see \cref{def_cox_complex}).
\end{remark}

\section{Coxeter coordinates for the space of barcodes} \label{sec_cox_coord}

\subsection{Describing $\mathcal{B}_n$ as a quotient}\label{sec_alt_description}
In this section, we describe $\mathcal{B}_n$ as a subset of a quotient of $\mbR^{2n}$. This will be used in the next section to equip this space with Coxeter complex coordinates.

Let $\targetspace \coloneqq \Sym \backslash \mbR^n \times \mbR^n $, where $\Sym$ acts diagonally by permuting the coordinates, i.e.~for $\gamma\in \Sym$, we set 
\begin{equation*}
    \gamma \cdot (x_1,\ldots, x_n, y_1,\ldots , y_n) = (x_{\gamma^{-1}(1)},\ldots, x_{\gamma^{-1}(n)}, y_{\gamma^{-1}(1)},\ldots , y_{\gamma^{-1}(n)}).
\end{equation*}
The elements of $\targetspace$ are equivalence classes of tuples $(x_1,\ldots, x_n, y_1,\ldots , y_n) \in \mbR^n \times \mbR^n$, which are denoted by $[x_1,\ldots, x_n, y_1,\ldots , y_n]$.

\begin{remark}
We write $\targetspace \coloneqq \Sym \backslash \mbR^n \times \mbR^n $ to emphasise that $\Sym$ acts from the left on this space. The reason we stress this is that later on, we will combine the statements here with descriptions of the Coxeter complex. There, the simplices are given by cosets $\tau P$ and the symmetric group acts on them by \emph{left} multiplication.
\end{remark}

There is a map $\phi$ from the space of barcodes with $n$ bars to $\targetspace$ given by
\begin{align*}
    \phi: \mathcal{B}_n &\to X = \Sym \backslash \mbR^n \times \mbR^n \\
\ls (b_i,d_i) \rs_{i \in \{1,..., n\}} &\mapsto [b_1,\ldots, b_n, d_1,\ldots, d_n].
\end{align*}
The image of $\phi$ is independent of the choice of indices for the bars of the barcode because the action of $\Sym$ is factored out. The map $\phi$ is clearly injective, but it is not surjective as the birth time of a homology class is always smaller than its death time. The image of $\phi$ is the subspace $\subtargetspace$ of $\targetspace$ given by 
\begin{equation*}
     \subtargetspace \coloneqq \Sym\backslash \ls (x_1,\ldots, x_n, y_1,\ldots , y_n) \in \mbR^n \times \mbR^n \,\middle|\, x_i < y_i \, \forall \,i \rs . 
\end{equation*}
For later reference, we note this observation in the following.
\begin{proposition}
\label{prop_bijection_B_n}
The map $\phi$ defines a bijection $\mathcal{B}_n\to \subtargetspace \subset \Sym \backslash \mbR^n \times \mbR^n$.
\end{proposition}

In \cref{sec_metric_B_n}, we equip $\mathcal{B}_n$ with metrics inspired by the bottleneck and Wasserstein distances. The map $\phi$ is an isometry with respect to these metrics.

\subsection{Coxeter complexes for birth and death}\label{sec_cox_coords_birthdeath}

We now introduce the Coxeter complex coordinates for $\mathcal{B}_n$. These coordinates are obtained by applying the map $(\restr{p}{\mbR^n \setminus L},q)$ of \cref{prop_CoxCoordsRn} to the two copies of $\mbR^n$ in $Y$.

\begin{theorem}
\label{thm_CoxCoords}
Every barcode $\ls (b_i,d_i) \rs_{i \in \{1,..., n\}} \in \mathcal{B}_n$ such that at least two of the $b_i$ and two of the $d_i$ are different from each other determines the following five data:
\begin{enumerate}
    \item its \emph{average birth time} $\bar{b} = \sum_{i=1}^n b_i /n \in \mbR$;
    \item its \emph{average death time} $\bar{d} = \sum_{i=1}^n d_i / n \in \mbR$;
    \item its \emph{birth standard deviation} $\lVert v_b \rVert = \left(\sum_{i=1}^n |b_i-\bar{b}|^2\right)^{1/2} \in \mbR_{> 0}$;
    \item its \emph{death standard deviation} $\lVert v_d \rVert = \left(\sum_{i=1}^n |d_i-\bar{d}|^2\right)^{1/2} \in \mbR_{> 0}$;
    \item an orbit $\Sym \cdot (b_\theta, d_\theta) \in \Sym  \backslash \CoxComplex \times \CoxComplex.$
\end{enumerate}
Furthermore, these five data uniquely determine $B$.
\end{theorem}

\begin{proof}
Let $B=\ls (b_i,d_i) \rs_{i \in \{1,..., n\}}$ be such that at least two $b_i$ and two $d_i$ are different. By assumption, both $(b_1,\ldots, b_n)$ and $(d_1,\ldots, d_n)$ are points in $\mbR^n \setminus L$. The image of $B$ under $\phi$ (\cref{prop_bijection_B_n}) is $$\phi(B)= [b_1,...,b_n,d_1,...,d_n] \in \Sym \backslash ( \mbR^n \setminus L \times \mbR^n \setminus L).$$
Since the map $(\restr{p}{\mbR^n \setminus L},q)$ is $\Sym$-equivariant (\cref{prop_CoxCoordsRn}), it induces a bijection 
\begin{equation*}
    \Sym \backslash \big( \mbR^n \setminus L \times \mbR^n \setminus L \big) \cong \Sym \backslash \big(\mbR \times \mbR_{>0} \times \CoxComplex) \times  (\mbR \times \mbR_{>0} \times \CoxComplex)\big).
\end{equation*}
The image of $[b_1,...,b_n,d_1,...,d_n]$ under this bijection is the $\Sym$-orbit of
$$(\restr{p}{\mbR^n \setminus L},q)^2(b_1,...,b_n,d_1,...,d_n)= (\bar{b},\lVert v_{b} \rVert,b_\theta, \bar{d},\lVert v_{d} \rVert,d_\theta) .$$

The claim now follows since the action of $\Sym$ on $(\bar{b},\lVert v_{b} \rVert,b_\theta, \bar{d},\lVert v_{d} \rVert,d_\theta)$ is trivial on $\bar{b}$, $\lVert v_{b} \rVert$, $\bar{d}$, $\lVert v_{d} \rVert$ and is given by the action of $\Sym$ on the Coxeter complex $\CoxComplex$ for $b_\theta,d_\theta$.
\end{proof}

\subsection{A stratification of $\mathcal{B}_n$}\label{sec_stratification}
In this section, we describe the stratification that we obtain from the description of $\mathcal{B}_n$ in terms of Coxeter complexes.

We start by extending \cref{def_permutation}, the permutation assigned to a strict barcode, to the general case of $\mathcal{B}_n$. For non-strict barcodes, we cannot uniquely assign a permutation. However, there is a nice description of the set of all possible such permutations in terms of double cosets of parabolic subgroups:
\begin{definition}\label{def_para_subgroups}
For a barcode  $B = \ls (b_i,d_i) \rs_{i \in \{1,..., n\}} \in \mathcal{B}_n$, let $\tau_b$ and $\tau_d$ be elements of $\Sym$ such that
$b_{\tau_b(1)} \leq \ldots \leq b_{\tau_b(n)}$ and $d_{\tau_d(1)} \leq \ldots \leq d_{\tau_d(n)}$. Let
\begin{equation*}
    P_b^B = \ls \gamma\in \Sym \,\middle| \, b_{\tau_b(i)} = b_{ \tau_b \gamma (i)} \rs , \,
     P_d^B =  \ls \gamma\in \Sym \,\middle| \, d_{\tau_d(i)} = d_{ \tau_d \gamma (i)} \rs.
\end{equation*}
The \emph{double coset $D_B$ associated to $B$} is defined as $D_B \coloneqq P_b^B \tau_b^{-1}\tau_d P_d^B$.
\end{definition}

\begin{remark}
\label{rem_double_coset_well_def}
Note that while $\tau_b$ and $\tau_d$ depend on the ordering of the barcode, $P_b^B$ and $P_d^B$ do not.
The groups $P_b^B$ and $P_d^B$ are parabolic subgroups of $\Sym$, as was observed in \cref{rem_uniquely_det}.
The cosets 
\begin{equation*}
    \tau_b P_b^B = \{ \rho \in \Sym \mid b_{\rho(1)} \leq \ldots \leq b_{\rho(n)}\}
\end{equation*}
 and 
\begin{equation*}
    \tau_d P_d^B = \{ \rho \in \Sym \mid d_{\rho(1)} \leq \ldots \leq d_{\rho(n)}\},
\end{equation*}
which are the sets of permutations that preserve the order of the $b_i$ and $d_i$ respectively, do not depend on the indexing of $B$ either. Hence, the double coset $D_B = (\tau_b P_b^B)^{-1} \cdot \tau_d P_d^B $ is indeed an invariant of the barcode $B$. 
Furthermore, if $B$ is a strict barcode, then $P_b^B = \ls \id \rs = P_d^B$, so $D_B = \ls  \tau_b^{-1}\tau_d  \rs = \ls \sigma_B \rs$ and we recover the definition of \cite{TRN} as given in \cref{def_permutation}.
\end{remark}

\begin{example}
Let $$B= \{(b_1,d_1)=(1,10), (b_2,d_2)=(2,5), (b_3,d_3)= (4,5), (b_4,d_4)= (4,7)\} \in \mathcal{B}_4.$$ One has $b_1 < b_2 < b_3 = b_4$ and $d_2 = d_3 < d_4 < d_1$. Let $\tau_b = [1234]$ and $\tau_d = [2341]$. They satisfy $b_{\tau_b(1)} \leq \ldots \leq b_{\tau_b(4)}$ and $d_{\tau_d(1)} \leq \ldots \leq d_{\tau_d(4)}$ respectively, but so do $\tau_b'=[1243]$ and $\tau_d'=[3241]$. In this case, one has $P_b^B= \{\id, (34)\}$, $P_d^B=  \{\id, (12)\} $ and $\tau_b P_b^B =\{ [1234], [1243]\} $, $\tau_d P_d^B = \{[2341], [3241] \}$. 
The double coset
\begin{align*}
D_B
& = \{\gamma_b \tau_b^{-1}\tau_d \gamma_d \mid \gamma_b \in P_b^B, \gamma_d \in P_d^B\} 
\\
& = \{ \tau_b^{-1}\tau_d, \tau_b'^{-1}\tau_d, \tau_b^{-1}\tau_d' ,\tau_b'^{-1}\tau_d'\} \\
&=
 \{ [2341],[2431],[3241],[4231]\} 
\end{align*}
is the set of all the permutations $\sigma$ that satisfy that the $j$-th death (in increasing order) is paired with the $\sigma(j)$-th birth.
\end{example}

Recall that the Coxeter complex $\CoxComplex$ is a simplicial complex with simplices given by cosets of parabolic subgroups $\tau P$. This simplicial decomposition gives it the structure of a stratified space over the poset of cosets of parabolic subgroups equipped with reverse inclusion (see \cref{sec_cox_background}).
Taking the cone and products of these simplices yields a decomposition of 
\begin{equation}\label{eq_dec_Rn}
  \mbR^{2n} \cong \operatorname{cone}(\CoxComplex) \times \mbR \times \operatorname{cone}(\CoxComplex) \times \mbR  
\end{equation}
into strata that are compatible with the action of $\Sym$, i.e.~each stratum is sent to another stratum of same dimension by the action of $\Sym$. This follows from \cref{rem_cone} and the fact that $\CoxComplex$ is stratified and the map $(\restr{p}{\mbR^n \setminus L},q)$ of \cref{prop_CoxCoordsRn} is $\Sym$-equivariant. 
The strata in \cref{eq_dec_Rn} are indexed by pairs of cosets $(\tau_1 P_1, \tau_2 P_2),$ where $\tau_1,\tau_2 \in \Sym$ and $P_1,P_2\leq \Sym$ are parabolic subgroups\footnote{
Note that, following \cref{rem_uniquely_det}, the points in $\textup{Conf}_n(\mbR) \times \textup{Conf}_n(\mbR)  \subset \mbR^n \times \mbR^n$ are exactly the ones that belong to the top-dimensional strata. 
Similarly, the points of $L \times L \subset \mbR^n \times \mbR^n$ belong to the lowest dimensional strata, corresponding to the cone points in \cref{eq_dec_Rn}.
}. The partial ordering on these pairs is given component-wise by reverse inclusion (cf.~\cref{eq_poset_CoxComplex}).

It follows that the quotient $X = \Sym \backslash \mbR^{2n}$ is stratified over the quotient $\mcP$ of this poset by the action of $\Sym$. More concretely, $\mcP$ can be described as follows: The elements of $\mcP$ are orbits of the form $\Sym \cdot (\tau_1 P_1, \tau_2 P_2)$, where $\tau_1,\tau_2 \in \Sym$ and $P_1,P_2\leq \Sym$ are parabolic subgroups. The partial ordering is given by 
\begin{equation*}
    \Sym\cdot (\tau_1 P_1, \tau_2 P_2) \leq \Sym\cdot (\tau'_1 P'_1, \tau'_2 P'_2)
\end{equation*}
if there is $\gamma \in \Sym$ such that 
\begin{equation*}
    \tau_1 P_1 \supseteq \gamma \tau'_1 P'_1 \text{ and } \tau_2 P_2 \supseteq \gamma \tau'_2 P'_2.
\end{equation*}
This quotient poset $\mcP$ has a more explicit description in terms of another poset $\mcQ$, which consists of ``marked'' double cosets of parabolic subgroups:

\begin{definition}
Let $\mcQ$ be the poset consisting of all triples $(P_1,P_1 \sigma P_2,P_2)$, where $\sigma\in \Sym$ and $P_1,P_2\leq \Sym$ are parabolic subgroups and where
\begin{equation*}
    (P_1,P_1 \sigma P_2,P_2) \leq (P_1',P_1' \sigma P_2',P_2')
\end{equation*}
if and only if there is component-wise containment in the reverse direction, 
\begin{equation*}
    P_1\supseteq P_1', \, P_2\supseteq P_2' \text{ and } P_1 \sigma P_2 \supseteq P_1' \sigma P_2'.
\end{equation*}
\end{definition}
A very similar poset is also studied as a two-sided version of the Coxeter complex by Hultman \cite{Hultman2007} and Petersen \cite{Pet:twosidedanalogue}. 
We remark that $\mcQ$ is different from the poset of all double cosets of the form $P_1 \sigma P_2$: There can be $P_1\not = P_1', P_2\not = P_2'$ such that $P_1\sigma P_2 = P_1'\sigma P_2'$ (see \cite[Remark 4]{Pet:twosidedanalogue}).

\begin{lemma}
\label{isom_posets}
The map 
\begin{align*}
    \phi: \mcP &\to \mcQ \\
    \Sym\cdot (\tau_1 P_1, \tau_2 P_2) &\mapsto (P_1,P_1\tau_1^{-1}\tau_2 P_2,P_2)
\end{align*}
is an isomorphism of posets. 
\end{lemma}
\begin{proof}
To see that $\phi$ is a bijection of the underlying sets, consider the following map:
\begin{align*}
    \psi: \mcQ &\to \mcP \\
    (P_1,P_1 \sigma P_2,P_2)& \mapsto \Sym\cdot (P_1, \sigma P_2).
\end{align*}
It is easy to verify that $\phi$ and $\psi$ are independent of the choices of representatives and are inverse to one another.
That $\phi$ is indeed a map of posets, i.e.~that it preserves the partial ordering, follows from elementary manipulations of cosets.
\end{proof}

\begin{theorem}\label{cor_stratification}
The set $\mathcal{B}_n$ of barcodes with $n$ bars is stratified over the poset $\mathcal Q$. The lowest dimensional stratum containing the barcode $B$ is the stratum corresponding to $(P_b^B, D_B, P_d^B)\in \mathcal{Q}$. It is of the form 
\begin{equation*}
    \mathcal{B}_n^{(P_b^B, D_B, P_d^B)} =
    \left(\Sym \cdot (\operatorname{cone}(\tau_b P_b^B) \times \mbR \times \operatorname{cone}(\tau_d P_d^B) \times \mbR)\right) \cap Y.
\end{equation*}
\end{theorem}

\begin{proof}
Recall that $\mathcal{B}_n \cong Y$ is a subset of $X = \Sym\backslash \mbR^{2n}$  (\cref{prop_bijection_B_n}).
As observed above, $X$ is stratified over the poset $\mcP$ and, by \cref{isom_posets}, this poset is isomorphic to $\mcQ.$ It follows that $\mathcal{B}_n$ is also stratified over $\mcQ$. The strata are obtained by taking the intersection with $Y$.

This stratification is induced by the simplicial structure of the Coxeter complexes in 
\begin{equation*}
    X \cong \Sym \backslash \big( \operatorname{cone}(\CoxComplex) \times \mbR \times \operatorname{cone}(\CoxComplex) \times \mbR \big).
\end{equation*}
Hence, the strata that contain a barcode $B\in \mathcal{B}_n$ only depend on the coordinate $\Sym \cdot (b_\theta, d_\theta) \in \Sym  \backslash \CoxComplex \times \CoxComplex$ that $B$ determines by \cref{thm_CoxCoords}. As explained in \cref{rem_uniquely_det}, the associated points $b_\theta,\, d_\theta \in \CoxComplex$ lie in the interior of the simplices $\tau_b P_b^B, \, \tau_d P_d^B$. Hence, the lowest dimensional stratum that contains $B$ corresponds to the $\Sym$-orbit of $(\tau_b P_b^B, \tau_d P_d^B)$. 
\end{proof}

Let $B$ be a strict barcode, that is, $b_i \neq b_j$ and $d_i \neq d_j$ for $i \neq j$. Then $B$ is contained in the top-dimensional stratum 
\begin{equation*}
    \mathcal{B}_n^{(\ls \id \rs, \ls \id \rs  \tau_b^{-1}\tau_d \ls \id \rs, \ls \id \rs)} =
    \left(\Sym \cdot (\operatorname{cone}(\tau_b\ls \id \rs) \times \mbR \times \operatorname{cone}(\tau_d\ls \id \rs) \times \mbR)\right)\cap Y.
\end{equation*}
Changing the representative of the $\Sym$-orbit, this can be rewritten as
\begin{equation*}
    \mathcal{B}_n^{(\ls \id \rs, \ls \sigma_B \rs, \ls \id \rs)} =
    \left(\Sym \cdot (\operatorname{cone}(\ls \id \rs) \times \mbR \times \operatorname{cone}(\sigma_B\ls \id \rs) \times \mbR)\right)\cap Y,
\end{equation*}
where $\sigma_B = \tau_b^{-1}\tau_d$ is the permutation associated to $B$ as in \cref{def_permutation}. In particular, the strata containing strict barcodes are in one-to-one correspondence with the elements of $\Sym$.

When one considers the cone and real line parameters in the stratification of \cref{cor_stratification}, one obtains regions that are determined by the averages and standard deviations of \cref{thm_CoxCoords} and by parabolic subgroups. 

\begin{corollary}\label{cor_simplified_coordinates}
The Coxeter coordinates of \cref{thm_CoxCoords} decompose the space $\mathcal{B}_n$ of barcodes with $n$ bars into disjoint regions. The region containing the barcode $B = \ls (b_i,d_i) \rs_{i \in \{1,..., n\}} \in \mathcal{B}_n$ is defined as the set of all barcodes $B'$ such that:
\begin{enumerate}
    \item its average birth time is the same as that of $B$, i.e.~$\bar{b}'= \bar{b}$;
    \item its average death time is the same as that of $B$, i.e.~$\bar{d}'= \bar{d}$;
    \item its birth standard deviation is the same as that of $B$, i.e.~$\lVert v_{b'} \rVert =\lVert v_{b} \rVert$;
    \item its death standard deviation is the same as that of $B$, i.e.~$\lVert v_{d'} \rVert =\lVert v_{d} \rVert$;
    \item $P_b^{B'} = P_b^B $, $P_d^{B'} = P_d^B $ and $D_B = D_{B'}$.
\end{enumerate}
For strict barcodes, the information of the last Item 5 is equivalent to specifying $\sigma_B$, the permutation associated to barcodes in \cref{def_permutation}.
\end{corollary}

\section{A metric on $\mathcal{B}_n$}
\label{sec_metric_B_n}
In this section, we explain how the description of $\mathcal{B}_n$ given in \cref{sec_alt_description} with $\mbR^n$ equipped with the $l^\infty$-norm gives rise to a naturally defined metric $\tilde{d}_B$ on $\mathcal{B}_n$ that is closely related to the bottleneck distance. Similarly, the $l^2$-norm on $\mbR^n$ leads to a modified Wasserstein distance $\tilde{d}_W$ on $\mathcal{B}_n$.

To describe $\tilde{d}_B$, we equip $\mbR^{2n}$ with the metric $d_\infty$ induced by the $l^\infty$-norm.
This metric induces a map $\targetspace \times \targetspace \to \mbR$ on the quotient by taking the minimum value over all representatives of the corresponding equivalence classes:
\begin{align}
\label{eq_distance_quotient}
\begin{split}
    d: \targetspace \times \targetspace &\to \mbR \\
    \big([x,y],[x',y']\big)&\mapsto \min_{\substack{(\tilde{x},\tilde{y}) \in [x,y], \\  (\tilde{x}', \tilde{y}')\in [x',y']}}  d_\infty(\,(\tilde{x},\tilde{y}),(\tilde{x}',\tilde{y}')\,).
\end{split}
\end{align}

We will show that this map restricted to $Y$ agrees with a modified version of the bottleneck distance.

\begin{definition}\label{modified_bottleneck}
Let $B= \{(b_i,d_i)\}_{i \in \{1,...,n\}}$ and $B'= \{(b'_i,d'_i)\}_{i \in \{1,...,n\}}$ be two barcodes in $\mathcal{B}_n$. The \emph{modified bottleneck distance} between $B$ and $B'$ is 
\begin{equation*}
    \tilde{d}_B(B,B') \coloneqq \min_{\gamma  \in \Sym} \max_{i \in \{1,...,n\}} \lVert (b_i,d_i) -(b'_{\gamma(i)},d'_{\gamma(i)}) \rVert_\infty.
\end{equation*}
where $\lVert \cdot \rVert_\infty$ is the $l^\infty$-norm on $\mbR^2$.
\end{definition}
Note that the difference between the modified bottleneck distance and the original bottleneck distance as defined in \cref{bottleneck} is that for the modified version, one does not allow to match points of the barcodes to the diagonal $\Delta$ (see \cref{fig_bottleneck}).
Furthermore, $\tilde{d}_B(B,B')$ is well-defined only if both $B$ and $B'$ contain the same number of bars, i.e.~if they are both elements of the same $\mathcal{B}_n$. This is not necessary for the definition of the regular bottleneck distance, cf.~\cref{rem_bottleneck-more-bars}.
\begin{proposition}
\label{prop_isometry_B_n}
The map $d$ defines a metric on $Y$ with respect to which $\phi:(\mathcal{B}_n, \tilde{d}_B) \longrightarrow (\subtargetspace, d)$ is an isometry.
\end{proposition}
\begin{proof}
As observed before in \cref{prop_bijection_B_n}, $\phi$ maps $\mathcal{B}_n$ bijectively onto $\subtargetspace$. Hence, it is sufficient to show that for arbitrary barcodes $B$ and $B'$,
\begin{equation*}\label{eq_distances}
    \tilde{d}_B(B,B') = d(\phi(B),\phi(B')).
\end{equation*}

This follows from simply spelling out the definitions.
For points $(x,y)$ and $(x',y')$ in $\mbR^n\times \mbR^n$,
\begin{align*}
       d_\infty((x,y), (x',y')) &= \max\ls |x_1-x'_1|, \ldots, |x_n-x'_n|,  |y_1-y'_1|, \ldots, |y_n-y'_n| \rs \\
     &= \max_{i=1,\ldots, n} \max\ls |x_i-x'_i|, |y_i-y'_i| \rs \\
     &= \max_{i=1,\ldots, n} \lVert (x_i,y_i) -(x'_i,y'_i) \rVert_\infty,  
\end{align*}
where $\lVert \cdot \rVert_\infty$ is the $l^\infty$-norm on $\mbR^2$.
Combining this with the definition of $d$ on $X$ (see \cref{eq_distance_quotient}), we obtain
\begin{align*}
    d(\phi(B),\phi(B')) &= \min_{\gamma\in\Sym} d_\infty(\,\phi(B),\gamma \cdot \phi(B')\,)\\
        &= \min_{\gamma\in\Sym} \max_{i=1,\ldots, n} \lVert (b_i,d_i) -(b'_{\gamma^{-1}(i)},y'_{\gamma^{-1}(i)}) \rVert_\infty.
\end{align*}
This is the same as the modified bottleneck distance of \cref{modified_bottleneck}.
\end{proof}

Similarly, starting with $\mbR^{2n}$ equipped with the $l^2$-norm, one can establish an isometry between $Y$ and $\mathcal{B}_n$ equipped with a modified Wasserstein distance instead.

\begin{figure}[H]
    \centering
    \includegraphics[scale = 0.4]{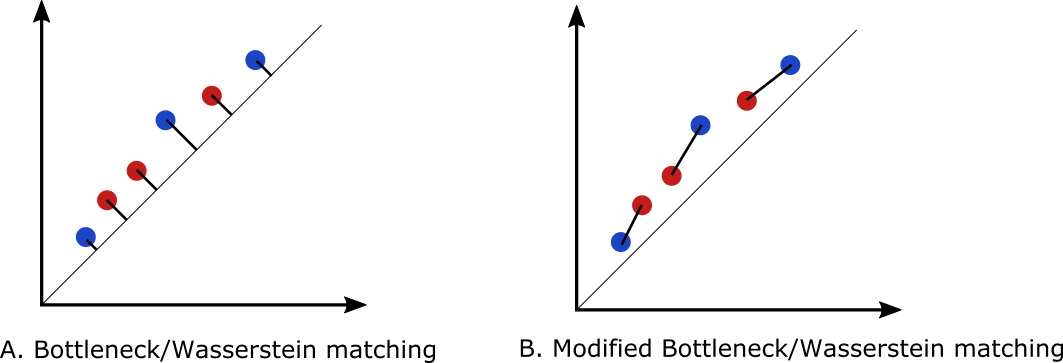}
    \caption{Two barcodes (red and blue) represented as persistence diagrams in $\mbR^2$. A. The matching that minimises the bottleneck or Wasserstein distance matches all the bars to the diagonal, as they are all very close to it. B. If bars are not allowed to be matched with the diagonal, the matching that minimises $\lVert (b_i,d_i) -(b'_{\gamma(i)},y'_{\gamma(i)}) \rVert_\infty$ for the bottleneck distance or $\sum_i \lVert (b_i,d_i) -(b'_{\gamma(i)},y'_{\gamma(i)}) \rVert_2$ respectively for the Wasserstein distance is different. }
    \label{fig_bottleneck}
\end{figure}

\begin{remark}\label{rem_bottleneck-more-bars}
Forgetting about the diagonal as done above opens the door to defining new metrics on barcodes by considering distances on $\mbR^n \times \mbR^n$ and then taking the quotient as was done in this section. It could potentially be extended to barcodes with different number of bars.
One could for instance imagine a map that forces matchings between as many bars as possible and then adds a positive weight equal to their distance to the diagonal to the unmatched bars if there are any. This is different from the bottleneck distance (or Wasserstein distance), which allows as many matchings as needed with the diagonal, see \cref{fig_bottleneck}.
When using barcodes to study data, bars close to the diagonal are usually considered as related to noise. However, there are cases where all the bars matter, for instance when the barcode is the one of a merge tree \cite{TRN,trees_barcodesII}. In such a case, a new metric that does not take the diagonal into account could turn out useful. We leave this for future work.
\end{remark}

\section{Future directions}
In this paper, we showed that the space $\mathcal B_n$ of barcodes with $n$ bars is stratified over the poset of marked double cosets of parabolic subgroups of $\Sym$. A question that arises is how this could be extended to the whole space of barcodes, i.e.~to the union $\bigcup_{n\in \mathbb{N}} \mathcal B_n$. An approach here would be to use appropriate inclusions $\mathcal{B}_m \hookrightarrow \mathcal{B}_{n}$ for $m\leq n$. Note that on the group level, there are natural injections $\Sym[m]\hookrightarrow \Sym[n]$. On the level of simplicial complexes, $\CoxComplex$ also contains copies of $\CoxComplex[{\Sym[m]}]$ for $m\leq n$.

It was shown in \cite{TRN,trees_barcodesII} that the permutation $\sigma_B$ associated to a strict barcode $B$ gives nice combinatorial insight on the number of merge trees that have the same barcode. This number, called the tree-realisation number (TRN), is derived directly from the permutation. It can also be used to do statistics on barcodes. Our coordinates (\cref{cor_simplified_coordinates}) firstly extend this work to any (possibly non-strict) barcode and secondly return a finer invariant than just the permutation. A future direction would be to study this finer invariant defined by $(\bar{b}, \bar{d}, \Vert v_b\Vert, \Vert v_d \Vert, \sigma_B)$. It might be well-suited for studying statistical questions: The first four elements already have descriptions as averages and standard deviations. The behaviour of the permutation $\sigma_B$ could be studied using tools from permutation statistics, such as the number of inversions or descents.

In a different direction, the description of $\mathcal{B}_n$ in terms of Coxeter complexes allows to rephrase these combinatorial questions in more geometric terms. Using this geometric perspective might give new ways for studying invariants and statistics on barcodes.

It would be interesting to see if the geometric and combinatorial tools developed here can help to understand inverse problems in TDA as the ones in \cite{TRN, trees_barcodesII,curry2017fiber,leygonie2021fiber}.
Since the merge tree to barcode problem is related to the symmetric group \cite{TRN,trees_barcodesII}, it is also natural to ask whether the stratification that we obtain in \cref{cor_stratification} can be extended to the space of merge trees with $n$ leaves. 

Lastly, the modified bottleneck and Wasserstein distances seem to have a different behaviour than the usual ones. A deeper study of their properties and their potential extension to the space of barcodes (see \cref{rem_bottleneck-more-bars}) is a natural next step to consider. 

\section*{Acknowledgements}
The authors would like to thank Kathryn Hess and Darrick Lee for the fruitful discussions and their useful comments on the manuscript.

\bibliographystyle{hplain}
\bibliography{bibliography}
\end{document}